\documentclass[a4paper,12pt]{amsart}
\usepackage{amsmath}
\usepackage{amsfonts}
\usepackage{amssymb}
\usepackage{hyperref}
\usepackage{fourier}
\usepackage[a4paper,left=2cm,right=2cm,top=2.5cm, bottom=2cm]{geometry}
\usepackage[ddmmyyyy]{datetime}

\newtheorem{lemma}{Lemma}%[section]
\newtheorem{theorem}[lemma]{Theorem}
\newtheorem{corollary}[lemma]{Corollary}
\newtheorem{proposition}[lemma]{Proposition}
\newtheorem{conjecture}[lemma]{Conjecture}
\newcounter{rotcount}
\newtheorem{rot}{}[rotcount]
\newenvironment{proofof}{\noindent}{\hfill$\Box$\medskip}
\newenvironment{rotproof}{\noindent}{\hfill$\lozenge$\smallskip}

\newcommand{\ncont}{\nsubseteq}\newcommand{\cont}{\subseteq}
\renewcommand{\u}{\cup}\renewcommand{\i}{\cap}
\newcommand{\s}{^*}\newcommand{\del}{\backslash}
\newcommand{\emp}{\emptyset}\newcommand{\cl}{{\rm cl}}
\newcommand{\si}{{\rm si}}\newcommand{\co}{{\rm co}}

\newcommand{\lc}{\left\lceil}\newcommand{\rc}{\right\rceil}
\newcommand{\defin}{\textbf}
\newcommand{\C}{\mathcal{C}}
\newcommand{\F}{\mathcal{F}}

\title{A splitter theorem on 3-connected matroids}
\author{Jo\~ao Paulo Costalonga\\}\thanks{The author was partially supported by CNPq, grant 478053/2013-4}
\address{{\upshape joaocostalonga@gmail.com}\\
  Universidade Federal do Esp\'irito Santo\\
  Av. Fernando Ferrari, 514; Campus de Goiabeiras\\
  29075-910 - Vit\'oria - ES - Brazil
 }

\begin{document}
\begin{abstract}
We establish the following splitter theorem for graphs and its generalization for matroids: Let $G$ and $H$ be $3$-connected simple graphs such that $G$ has an $H$-minor and $k:=|V(G)|-|V(H)|\ge 2$. Let $n:=\left\lceil k/2\right\rceil+1$. Then there are pairwise disjoint sets $X_1,\dots,X_n\subseteq E(G)$ such that each $G/X_i$ is a $3$-connected graph with an $H$-minor, each $X_i$ is a singleton set or the edge set of a triangle of $G$ with $3$ degree-$3$ vertices and $X_1\cup\cdots\cup X_n$ contains no edge sets of circuits of $G$ other than the $X_i$'s. This result extends previous ones of Whittle (for $k=1,2$) and Costalonga (for $k=3$).
\end{abstract}
\maketitle

\section{Introduction}

For a $3$-connected matroid $M$ with an $N$-minor, we say that a set $X\cont E(M)$ is \defin{$N$-contractible} or \defin{vertically $N$-contractible} in $M$ if $M/X$ or $\si(M/X)$, respectively, is a $3$-connected matroid with an $N$-minor. We also define $x\in E(M)$ as \defin{$N$-contractible} or \defin{vertically $N$-contractible} in $M$ if $M/x$ or $\si(M/x)$, respectively, is a $3$-connected matroid with an $N$-minor. We say that $x\in E(M)$ is \defin{$N$-deletable} if $M\del x$ is $3$-connected with an $N$-minor. When $N$ is the empty matroid, we simply say that the element or set is \defin{(vertically) contractible} or \defin{deletable}, according to the suitable case.

Deletable and contractible elements are vastly used in matroid and graph theory as inductive tools. In one hand, there are results regarding the number of (vertically) contractible elements in matroids and graphs and their structure and distribution \cite{Ando, Egawa, Wu, Oxley-Wu, Oxley-Wu2}. In other hand, there are the so-called splitter theorems that asserts that, for $3$-connected matroids $M>N$ satisfying certain hypothesis, there is an (vertically) $N$-contractible or $N$-deletable element in $M$. For example, Seymour's Splitter Theorem~\cite{Seymour1980} and many others \cite{Bixby87, Kingan, Brettel, Zwan}.

We establish a theorem towards the unification of both families of results. The fundamental question we seek answer may be stated as follows: ``given $3$-connected matroids $M>N$ how many $N$-vertically contractible elements can guarantee to exist? What can we say about their distribution?''. In particular, we generalize the following theorem:

\begin{theorem}\label{whittle-teo}(Whittle~\cite{Whittle}($k=1,2$) and Costalonga~\cite{Costalonga2}($k=3$))
Let $k\in\{1,2,3\}$ and let $M$ be a $3$-connected matroid with a $3$-connected simple minor $N$ such that $r(M)-r(N)\ge k$. Then, $M$ has a $k$-independent set of vertically $N$-contractible elements.
\end{theorem}

%Such kind of theorem is useful when only one of the deletion or contraction is a friendly operation for one's purposes. Whittle~\cite{Whittle} established Theorem \ref{whittle-teo} for $k\le 2$ to study stabilizers of classes of representable matroids. As another example of application, Theorem \ref{whittle-teo} is used in a series of results about non-separating cocircuits and graphicness in binary matroids~\cite{Costalonga, Lai, Lemos2009}.

%\section{rest of intro}

The requrement of $N$ to be simple in the theorem above is necessary, for instance, $M\in\{M(K_4),U_{2,n}\}$ has an $U_{1,3}$-minor but no edge $e$ such that $\si(M/e)$ has an $U_{1,3}$-minor. Theorem \ref{whittle-teo} is not valid for larger values of $k$ with a similar statement. Indeed, $M\s(K_{3,n}''')$ has only three vertically contractible elements \cite[Theorem 2.10]{Wu}. However, we prove an extension of Theorem \ref{whittle-teo} considering another vertically $N$-contractible structure than a single element.

For a matroid $M$ and $n\ge 3$, a sequence of elements $K:=x_1,\dots,x_n,y_1,\dots,y_n$ is said to be an \defin{$N$-carambole} of $M$ if $L:=\{y_1,\dots,y_n\}$ is a vertically $N$-contractible line of $M$ with $n$ distinct elements and, for each $i\in[n]$, $(L-y_i)\u x_i$ is a cocircuit of $M$. In this case, we say that $L$ is the \defin{filament} and $X:=\{x_1\dots,x_n\}$ is the \defin{hull} of $K$. Note that in the graphic case a filament corresponds to a triangle whose vertices have degree $3$. We say that a family $\{X_1,\dots,X_n\}$ of subsets of $E(M)$ is a \defin{free} family of $M$ if $M|(X_1\u\cdots\u X_n)=M|X_1\oplus\cdots\oplus M|X_n$. Note that, if each $X_i$ is a singleton set, then such family is free if and only if $X_1\u\cdots\u X_n$ is independent in $M$. Our main theorem is the following.

\begin{theorem}\label{main-theorem}
If $M$ is a $3$-connected matroid with a $3$-connected simple minor $N$ such that $k:=r(M)-r(N)\ge 2$, then $M$ has a free family with cardinality $\lceil k/2\rceil+1$ whose members are vertically $N$-contractible singleton sets or $N$-filaments of $M$.\end{theorem}

For graphic matroids, we have:

\begin{corollary}\label{main-graphic}
Suppose that $G$ and $H$ are $3$-connected simple graphs, that $G$ has an $H$-minor and that $k:=|V(G)|-|V(H)|\ge 2$. Then, there is a family $\F:=\{X_1,\dots,X_n\}$ of pairwise disjoint subsets of $E(G)$, such that $n\ge\lceil k/2\rceil+1$, each edge set of a circuit of $G[X_1\u\cdots\u X_n]$ is a member of $\F$ and, for each $i\in[n]$:
\begin{enumerate}
\item [(a)] $X_i$ is a singleton set such that $G/X_i$ is $3$-connected with an $H$-minor, or
\item [(b)] $X_i$ is the edge set of a triangle of $G$ with three degree-$3$ vertices and $G/X_i$ is $3$-connected and simple with an $H$-minor.
\end{enumerate}
\end{corollary}

For a degree-$3$ vertex $v$ in a $3$-connected graph $G'$, up to labels, there is an unique way to build a $3$-connected graph $G$ such that $G/T=G'$ and $T$ is a triangle of $G$ whose adjacent edges in $G$ are the ones adjacent to $v$ in $G$. Similarly, for a non-trivial cosegment $X$ in a $3$-connected matroid $M'$, there is an unique matroid $M$ with a carambole with $X$ as hull and a filament $L$ such that $M/L=M'$ (up to the labels of the elements of $L$). In Section \ref{fancy} this claim is be proved and the relations between such $M$ and $M'$ are described.

The size of the family in Theorem \ref{main-theorem} is sharp. Even if we drop the requirement that the family obtained in the Theorem \ref{main-theorem} is free, the size $\lceil k/2\rceil+1$ cannot be improved. The sharp family of examples in Section \ref{sec-sharpness} also holds for this weaker version of the Theorem. 

In our studies, a structure weaker than a carambole raises naturally in the critical cases as an obstruction for some elements to be vertically $N$-contractible. An \defin{$N$-biweb} is a sequence of elements $x_1,x_2,y_1,y_2,y_3$ such that $\{y_1,y_2,y_3\}$ is a vertically $N$-contractible triangle and, for $i=1,2$, $\{x_i,y_{3-i},y_3\}$ is a triad of $M$. The following corollaries strengthen Theorem \ref{main-theorem} for $k=4,5$.

\begin{corollary}\label{k4}
If, in Theorem \ref{main-theorem}, $k=4$, and $M$ has no $4$-independent set of vertically $N$-contractible elements, then $M$ has a $3$-independent set of vertically contractible elements in an $N$-biweb.\end{corollary}

\begin{corollary}\label{k5}
If, in Theorem \ref{main-theorem}, $k=5$, then $M$ has a $4$-independent set of vertically $N$-contractible elements or an $N$-filament with $3$ elements.
\end{corollary}

Let $M$ be a $3$-connected graphic matroid other that a wheel. If $F^+$ is a maximal fan of $M$ with respect to having its extremes in triads of $M$, then we say that the set $F$ of the non-extreme elements of $F$ is an inner fan of $M$. Let $N$ be a $3$-connected minor of $M$. Costalonga~\cite{Costalonga-Contractible} established that $M$ has a free family whose members are vertically $N$-contractible inner fans or singleton sets and whose sum of the rank of the members is $r(M)-r(N)$. A problem to be considered in further investigations is to establish a generalization of this result to non graphic matroids. Specially, in such a way that the next conjecture follows as a corollary:

\begin{conjecture}
If $M$ is a triangle-free $3$-connected matroid with a simple $3$-connected minor $N$, then $M$ has an independent set $I$ of $N$-contractible elements such that $|I|=r(M)-r(N)$.
\end{conjecture}

We follow the terminology of Oxley~\cite{Oxley}. The symbol ``$\lozenge$'' is used to indicate the end of a nested proof. We also denote $[n]:=\{1,\dots,n\}$ and use the same letter to refer to a sequence of elements and the set of its elements. Some notations will remain fixed since the beginning of Section \ref{fancy}.

\section{Sharpness}\label{sec-sharpness}

A sharp case for Theorem \ref{main-theorem} is constructed next. For $n\ge 4$, let $K$ be a copy of $K_{3,n}'''$ and $U:=\{v\in V(K): d_K(v)>3\}$. For $m\ge 4$, let $H$ be a copy of the bipartite graph $K_{m,m}$ with stable classes of vertices $A$ and $B$, $U\cont B$ and $V(H)\i V(K)=U$. Let $G$ be the union of $H$ and $K$ and define $M:=M\s(G)$ and $N=M\s(H)$. Denote $V(K)-U=\{v_1,\dots,v_n\}$. Let $X_i$ be the set of edges incident to $v_i$ in $G$. Note that each $X_i$ is the filament of carambole of $M$ whose other edges are in $K[U]$.

We will prove that, if $G\del e$ has an $H$-minor, then $e\in E(K)$. Indeed, consider $I,J\s\cont E(G)$ such that $H\cong G/I\del J\s$, $e\in J\s$, $I$ is independent and $J\s$ is coindependent in $M(G)$. Note that $G\del J\s$ has no isolated vertices. Since $H$ has no degree-$3$ vertices and $G\del J\s$ has no isolated vertices, then, each $v_i$ is incident to an element of $I$. But $\{v_1,\dots,v_n\}=V(K)-V(H)$. So, $|I|=n$. We may choose, in a natural way, $V(G/I)=V(H)$. If $e\in E(H)$, the vertex of $A$ incident to  $e$ has degree $m-1$ in $G/I\del e$. But $H$ is $m$-regular and $V(H)=V(G/I)$. A contradiction. This implies that $e\notin E(H)$ and, therefore $e\in E(K)$.

So, each family satisfying Theorem \ref{main-theorem} for $M$ and $N$ has all members contained in $E(K)$ and, therefore, $\F:=\{X_1,\dots,X_n\}\u \{\{e\}:e\in E(G[U])\}$ is a maximum sized such family. Note that $k:=r(M)-r(N)=2n+3$. Hence, $|\F|=n+3=\lc\frac{2n+6}{2}\rc=\lc \frac k2\rc+1$ and the desired sharpness holds.

\section{Preliminaries}\label{preliminaries}

In this section we establish some preliminary Lemmas. 

\begin{lemma}\label{rank3}
If $H$ is a connected rank-$3$ simple matroid and $a,b\in E(H)$, then there is a $4$-circuit of $H$ containing $a$ and $b$ or $H$ is the parallel connection of two lines with base point $a$ or $b$.
\end{lemma}
\begin{proof}
First consider the case that $a\neq b$. Suppose that the result does not hold in such case. Hence, $H$ has a triangle $T$ containing $a$ and $b$. Since $H$ has rank $3$ and no coloops, there are distinct elements $c$ and $d$ in $E(H)-\cl_H(T)$. The dependent set $\{a,b,c,d\}$ is not a circuit, so we may assume that $S:=\{b,c,d\}$ is a triangle of $H$. There is an element $e\in E(M)-(\cl_H(T)\u\cl_H(S))$ because $H$ is not the parallel connection of $\cl_H(T)$ and $\cl_H(S)$. By assumption, the dependent set $\{a,b,d,e\}$ is not a circuit of $H$. So, it contains a triangle $R$. By construction, $e\notin \cl_H(S)$. Hence, $a\in R$. Also, $d,e\notin \cl_M(\{a,b\})$ and, therefore, $R=\{a,d,e\}$. By circuit elimination on $R$, $S$ and $d$, there is a circuit $C$ of $H$ contained in $(R\u S)-d$. Since $R-d$ and $S-d$ are in distinct lines of $H$, then $C=(R\u S)-d$, which is a $4$-circuit of $H$ containing $a$ and $b$, a contradiction. Thus, the result holds if $a\neq b$.

Now, for $a=b$, suppose that $H$ has no $4$-circuit containing $a$ and consider a basis $\{a,x,y\}$ of $M$. If $a$ is in no $4$-circuit of $M$, then applying the previous case to $a$ and $x$ and to $a$ and $y$, we conclude that $M$ is the parallel connection of two nontrivial lines with base point $a$ and the result holds in general.
\end{proof}

\begin{lemma}\label{vertsep}
Let $H$ be a vertically connected matroid and $Y\cont E(H)$. Suppose that $\{A,B\}$ is a vertical $2$-separation of $H$ with $A$ minimal with respect to containing $Y$. Then $(A-Y)\i \cl_H(B)=\emp$. Moreover, if $x\in (A-Y)\i \cl\s_H(B)$, then $r_M(A)=2$ and $x$ is a coloop of $H|A$.
\end{lemma}
\begin{proof}
First we will prove that $(A-Y)\i \cl_H(B)=\emp$. Suppose for a contradiction that $x\in (A-Y)\i \cl_H(B)$. Then $\lambda_H(A-x)\le \lambda_H(A)$. By the minimality of $A$, $r_H(A-x)<2$ and so $r_H(A-x)=1$ and $r_H(A)=2$. But this implies that $\lambda_H(A-x)\le \lambda_H(A)-1=0$. A contradiction to the vertical connectivity of $H$.

For the second part, consider $x\in (A-Y)\i \cl\s_H(B)$. Again, $\lambda_H(A-x)\le \lambda_H(A)$ and, therefore, $r_H(A-x)=1$. So, $r_H(A)=2$ and $x$ is a coloop of $H|A$.
\end{proof}

\begin{lemma}\label{2-cocirc}
Let $H$ be a vertically connected but not vertically $3$-connected matroid. Suppose that $z$ is an element of $H$ such that $H/z$ is vertically $3$-connected. If $A$ is a minimal vertical $2$-separating set of $H$ with respect to containing $z$, then $A$ is a rank-$2$ cocircuit of $H$.
\end{lemma}
\begin{proof}
Note that $z$ is not a loop of $H$. Let $B:=E(H)-A$. We may write $H=L\oplus_2 K$ with $E(L)=A\u p$, where $p$ is the base point of the $2$-sum. Since $r_{H/z}(A-z)$ and $r_{H/z}(B)$ are both at least $1$ and $H/z$ is vertically connected, then $z$ is not in parallel with $p$ in $L$. So, $z\notin \cl_H(B)$. Hence, by Lemma \ref{vertsep} for $Y=\{z\}$, $B$ is a flat of $H$. This implies that $r_{H/z}(B)=r_H(B)$ and, therefore, $\lambda_{H/z}(A)=\lambda_H(A)$. But $H/z$ is vertically $3$-connected, thus $r_{H/z}(A-z)\le 1$. Hence, $r_H(A)=2$. Thus, $r_H(B)=r(H)-1$. So, $B$ is an hyperplane and $A$ is a rank-$2$ cocircuit of $H$.
\end{proof}

The next Lemma has a straightforward proof.

\begin{lemma}\label{cocircuit-union}
If $M$ is a $3$-connected matroid, $r(M)\ge 4$ and $C\s$ and $D\s$ are distinct cocircuits of $M$, then $r_M(C\s\u D\s) \ge 4$.
\end{lemma}

\section{Caramboles and their properties}\label{fancy}

In this section, we present some attractive properties of caramboles. From this point, we will have some notations fixed as described next. We always consider $M$ as a $3$-connected simple matroid with a $3$-connected simple minor $N$. When talking about a carambole, by standard, we will denote it by $K=x_1,\dots,x_n,y_1,\dots, y_n$, its filament by $L$, its hull by $X$ and $C\s_i:=(L-y_i)\u x_i$.

\begin{proposition}\label{1st-prop}
If $X$ is the hull of a carambole of a $3$-connected matroid $M$ with $r(M)\ge 4$, then $r\s_M(X)=2$. Moreover, the filament of such carambole has same cardinality as $X$.
\end{proposition}

\begin{proof} By Lemma \ref{cocircuit-union}, $|X|=n$. Consider a cocircuit $C\s\cont (C\s_1\u C\s_2)-y_3$. Note that $C\s-L\cont \{x_1,x_2\}$. As $r_M(C\s)\ge 3$, $C\s$ meets $\{x_1,x_2\}$ and we may assume that $x_1\in C\s$. Moreover, $C\s$ meets $L$ and, by orthogonality, $L-y_3\cont C\s$. Now consider $D\s\cont (C\s\u C\s_3)-y_2$. Since $y_3,y_2\notin D\s$, then $D\s$ avoids $L$. Therefore, $D\s=\{x_1,x_2,x_3\}$. Analogously, $\{x_i,x_j,x_k\}$ is a triad of $M$ for each $3$-subset $\{i,j,k\}\cont[n]$ and the proposition holds.
\end{proof}

\begin{proposition}\label{prov}
Let $M$ be a $3$-connected matroid with $r(M)\ge 4$. Suppose that $K:=x_1,\dots,x_n,y_1,\dots,y_n$ is a carambole of $M$ with filament $L$ and hull $X$. Suppose that $C\in \C(M)$ and $C\ncont L$. Then the following assertions hold:
	\begin{enumerate}
		\item [(a)] If $C$ intersects $K$, then
			\begin{enumerate}
				\item [(a.1)] $X\cont C$ and $(C-L) \u A\in \C(M)$ for each $2$-subset $A$ of $L$, or
				\item [(a.2)] For some $l\in[n]$, $X-C=\{x_l\}$, $(C-L)\u y_l\in \C(M)$ and, for each $2$-subset $A$ of $L-y_l$, $(C-L)\u A\in\C(M)$.
			\end{enumerate}
		\item [(b)] $C-L\in \C(M/L)$.
		%\item [(c)] If $D$ is a circuit of $M/L\del X$, then $D$ is a circuit of $M$.
		%\item [(d)] If $N$ is a matroid having $K$ as carambole and $M/L=N/L$, then $M=N$.
	\end{enumerate}
\end{proposition}

\begin{proof}We prove first:

\begin{rot}\label{prov-1}
Let $\{i,j\}$ be a $2$-subset of $[n]$ and let $D$ be a circuit of $M$ such that $\{y_i\}\cont D\i L\cont\{y_i,y_j\}$. Then $|X-D|\le 1$ and one of the following alternatives holds:
\begin{enumerate}
\item [(i)] $D\i L=\{y_i\}$, $X-D=\{x_i\}$ and $D_1:=(D-L)\u\{y_j,y_k\}\in \C(M)$ for each $k\in[n]-\{i,j\}$; or
\item [(ii)] $D\i L=\{y_i,y_j\}$ and for each $k\in [n]-\{i,j\}$, one of the following holds:
	\begin{enumerate}
	 \item [(ii.1)]$X-D=\{x_k\}$ and $D_1:=(D-L)\u y_k\in \C(M)$, or
	 \item [(ii.2)]$X-D\neq\{x_k\}$ and $D_1:=(D-L)\u\{y_k,y_j\}\in \C(M)$.
	\end{enumerate}
\end{enumerate}
\end{rot}
\begin{rotproof}
Note that $y_i\in C\s_j\i D\cont \{x_j,y_i\}$. By orthogonality, $x_j\in D$. By Proposition \ref{1st-prop}, $r\s_M(X)=2$ and, by orthogonality with $D$, $|X-D|\le 1$. Let $k\in[n]-\{i,j\}$ and let $D_1$ be a circuit of $M$ with $x_j\in D_1\cont (D\u\{y_i,y_j,y_k\})-y_i$. Since $x_j\in D_1\i C\s_j\cont\{x_j,y_k\}$, thus, by orthogonality, $y_k\in D_1$. Let $D_2$ be a circuit of $M$ with $x_j\in D_2\cont (D_1\u\{y_i,y_j,y_k\})-y_k$. As $D_1\cont D\u\{y_j,y_k\}$, then $D_2\cont D\u \{y_i,y_j\}=D\u y_j$.

First, we consider the case that $D\i L=\{y_i\}$. By orthogonality with $C\s_i$, $x_i\notin D$. As $|X-D|\le 1$, $X-x_i\cont D$. As $D_2\cont D\u y_j$ and $x_i\notin D$, then $x_i\notin D_2$. Moreover, $|D_2\i L|\cont\{y_i,y_j\}$. So, $C\s_i\i D_2\cont \{y_j\}$ and, by orthogonality, $y_j\notin D_2$. So, $D_2\cont D$ and, therefore, $D_2=D$. Since $D_2-L\cont D_1-L\cont D-L=D_2-L$, then $D_1-L=D-L$. We already checked that $y_k\in D_1$. Since $x_k\in D-L=D_1-L$, then, by orthogonality with $C\s_k$, $y_j\in D_1$. So, $D_1=(D-L)\u \{y_j,y_k\}$ and (i) holds in this case.

So, assume that $D\i L=\{y_i,y_j\}$. As $D_2\cont D\u y_j=D$, hence $D_2=D$ and $D_1-L=D-L$ again. We already checked that $y_k\in D_1$. If $X-D=\{x_k\}$, as $D_1\i L\cont\{y_j,y_k\}$, then, by orthogonality with $C\s_k$, $y_j\notin D_1$. So, $D_1=(D-L)\u y_k$ and we have (ii.1). Otherwise, $X\cont D$ and, by orthogonality with $C\s_k$, $y_j\in D_1$ and (ii.2) holds. So, \ref{prov-1} holds.
\end{rotproof}

By orthogonality with the $C\s_i$'s, each circuit intersecting $K$ but not contained in $L$, intersects $L$ in $\{y_i\}$ or $\{y_i,y_j\}$ for some $i,j\in [n]$. Now, it is clear that (a) follows by applying \ref{prov-1} iteratively.

For item (b), if $D$ is a circuit of $M$ such that $\emptyset\subsetneq D-L\cont C-L$, by item (a), $(D-L)\u (C\i L)$ contains a circuit, but this set is contained in $C$, so it is $C$ and $C-L=D-L$. This finishes the proof.
\end{proof}

\begin{corollary}\label{triangle-carambole}
If $L$ is a filament of a $3$-connected matroid $M$ with $r(M)\ge 4$, then $M/L$ is $3$-connected. In particular, all triangles of $M$ meeting $K$ are contained in $L$
\end{corollary}
\begin{proof} By the definition of filament, $\si(M/L)$ is $3$-connected. So, it is enough to prove the second part of the corollary. Suppose for a contradiction that $T$ is a triangle of $M$ meeting $K$ but not contained in $L$. By Proposition \ref{prov}, $|L\i T|\ge 1$ and $|X\i T|\ge |X|-1$. Then $n=3$ and we may assume that $T=\{x_1,x_2,y_3\}$. Now, it follows that $r_M(C\s_1\u C\s_2)=3$. A contradiction to Lemma \ref{cocircuit-union}.
\end{proof}

The next Proposition establishes that $M$ can be rebuild in an unique way from $M/L$. 

\begin{proposition}\label{reconstruction}
Let $M$ and $H$ be $3$-connected matroids with rank at least four, both having a common carambole with filament $L$. Then $M=H$ if $M/L=H/L$.
\end{proposition}
\begin{proof}
By symmetry, it is enough to prove that $\C(H)\cont \C(M)$. Indeed, suppose that $D$ is a circuit of $H$ meeting $L$. So, $D-L$ is a circuit of $H/L=M/L$ by Proposition \ref{prov} (b) on $H$. Thus there is a circuit $C$ of $M$ such that $C-L=D-L$. Since $X-D=X-C$, then by Proposition \ref{1st-prop} (a) on $M$, it follows that $(C-L)\cup (D\cap L)=D$ is a circuit of $M$. Thus $\C(H)\cont \C(M)$.
\end{proof}

Such reconstruction of $M$ from $M/L$ is made in a more explicit way as follows. Let $X:=\{x_1,\dots,x_n\}$ and $L:=\{y_1,\dots,y_n\}$. Let $\Theta_n$ be the matroid on $X\u L$ such that $\Theta_n\s$ has $L$ as a line, $X$ as a coline and, for $i\in[n]$,  $C\s_i:=(L-y_i)\u x_i$ as a cocircuit, as defined by Oxley~[Proposition 11.5.1]\cite{Oxley} (see also \cite{Delta}). Then 

\begin{corollary}\label{deltawye}
Consider $L$ and $X$ as described above. Suppose that $H$ is a cosimple matroid with $X\cont E(H)$, $r\s_H(X)=2$ and $L\i E(H)=\emp$. Suppose also that $M$ is a $3$-connected simple matroid having $x_1,\dots,x_n,y_1,\dots,y_n$ as carambole. Then $H=M/L$ if and only if $M\s=P_X(H\s,\Theta_n)$.
\end{corollary}

Such reconstruction is the same as the one used in generalized delta-wye exchanges. In the next proposition, we see that the existence of filaments also guarantees the existence of certain independent sets of $N$-contractible elements. This proposition will be proved in Section \ref{sec-proofs}.

\begin{proposition}\label{independent-seeds}
Let $M$ be a $3$-connected matroid with an $N$-minor. Suppose that $r(M)\ge 4$. If $X$ is the hull of an $N$-carambole, then $X$ is an independent set of $M$ whose elements are $N$-contractible.
\end{proposition}

\section{Other structures and their properties}\label{section-structures}

In this section, we will define some structures and establish some of their properties. We keep the notations fixed in the beginning of Section \ref{fancy}.  We say that a line in a matroid is \defin{non-trivial} if it has at least three points.

A vertically contractible element of $\si(M/x)$ may be not vertically contractible in $M$. This is the greatest difficulty to apply an inductive strategy to our problem. Whittle \cite{Whittle} characterized the structures that may appear in such situation. We will describe such structures and strengthen their characterizations next.

An \defin{$(M,N)$-vertbarrier} is a pair $(C\s,p)$, where $C\s$ is a rank-$3$ cocircuit of $M$, $p\in \cl_M(C\s)-C\s$ and $si(M/x,p)$ is $3$-connected with an $N$-minor for some $x\in C\s$ (and therefore for all $x\in C\s$ as established ahead, in Lemma \ref{w37}). We say that $(C\s,p)$ \defin{contains} $x$ if $x\in C\s$. The next Lemma is a generalization of Lemma 3.6 of \cite{Whittle}.

\begin{lemma}\label{w36}Suppose that $x$ and $p$ are elements of $M$ such that $\{x,p\}$ is vertically $N$-contractible in $M$ but $p$ is not. Then $r(M)\ge 4$ and there is an $(M,N)$-vertbarrier $(C\s,p)$ containing $x$.
\end{lemma}%revised once
\begin{proof}
If $|E(M)|\le3$, the result is clear. Assume the contrary. So, $M/p$ is connected, and, therefore, vertically connected. Since $M/p$ is not vertically $3$-connected, then $r(M/p)\ge3$. Therefore, $r(M)\ge 4$. By Lemma \ref{2-cocirc} for $(H,z)=(M/p,x)$, $M/p$ has a rank-$2$ cocircuit $C\s$ containing $x$. Since $M$ is $3$-connected and $r(M)\ge 4$, hence $r_M(C\s)=3$. Moreover $p\in \cl_M(C\s)$, because $r_{M/p}(C\s)=2$.
\end{proof}

Whittle~\cite{Whittle} established the following two lemmas:

\begin{lemma}(Lemma 3.7 of \cite{Whittle})\label{w37}
Suppose that $C\s$ is a rank-$3$ cocircuit of $M$ and $p\in \cl_M(C\s)-C\s$.
\begin{enumerate}
\item[(a)] If $x,y\in C\s$, then  $\si(M/p,x)\cong \si(M/p,y)$.
\item[(b)] If, for some $z\in C\s$, $\{z,p\}$ is vertically $N$-contractible, then, for each $x\in C\s$, $\{x,p\}$ is vertically $N$-contractible.
\end{enumerate}
\end{lemma}%revised once

\begin{lemma}(Lemma 3.8 of \cite{Whittle})\label{w38}
If $C\s$ is a rank-$3$ cocircuit of $M$ and $x\in C\s$ has the property that $\cl_M(C\s)-x$ contains a triangle of $M/x$, then $\si(M/x)$ is $3$-connected.
\end{lemma}%revised once

Lemma \ref{w38} implies the two next corollaries:

\begin{corollary}\label{w38-cor}
If $M$ has a triangle $T$ intersecting a triad $T\s$, then, for $x\in T\s-T$ and $y\in T-T\s$, $\si(M/x)$ and $\co(M\del y)$ are $3$-connected.
\end{corollary}%revised once

\begin{corollary}\label{4-circuit}
Suppose that $C\s$ is a rank-$3$ cocircuit of $M$, then each element of $C\s$ in a $4$-circuit of $M|\cl_M(C\s)$ is vertically $N$-contractible in $M$.
\end{corollary}

\begin{lemma}\label{carambola potencial}
If, for $n\ge 3$, $L:=\{y_1,\dots,y_n\}$ is a line of $M$ and $x_1,\dots,x_n$ are elements of $M$ such that, for each $i\in[n]$, $C\s_i:=(L-y_i)\u x_i$ is a cocircuit of $M$, then $K:=x_1,\dots,x_n,y_1,\dots,y_n$ is a carambole of $M$. Moreover, if, for some distinct $i$ and $j$ in $[n]$, $M/x_i,y_i$ or $M/y_i,y_j$ has an $N$-minor, then $K$ is an $N$-carambole of $M$.
\end{lemma}
\begin{proof}
First we prove the Lemma for $n=3$. Since $y_3\in T-T\s_3$, then, by Corollary \ref{w38-cor}, $\co(M\del y_3)$  is $3$-connected. By orthogonality, $T\s_1$, $T\s_2$ and $T\s_3$ are the unique triads of $M$ meeting $T$. Hence, $\{x_1,y_2\}$ and $\{x_2,y_1\}$ are the unique serial pairs of $M\del y_3$. So, $M/T$ is $3$-connected because $M/T=M\del y_3/y_1,y_2 \cong \co(M\del y_3)$. Thus, $W$ is a carambole of $M$. Now, say that $M/x_1,y_1$ has an $N$-minor. By Lemma \ref{w37}, $\si(M/x_1,y_1)\cong\si(M/y_1,y_2)=\si(M/T)$. So, $M/T$ has an $N$-minor and the lemma holds for $n=3$.

Now, let us consider $n\ge 4$. Let  $1\le i<j<k\le n$ and $Y:=L-\{y_i,y_j,y_k\}$. So, $M\del Y$ is $3$-connected. By the lemma for $n=3$, $W:=x_i,x_j,x_k,y_i,y_j,y_k$ is a carambole of $M\del Y$. Thus, $M\del Y/\{y_i,y_j,y_k\}=M/L$ is $3$-connected. The first part of the lemma is proved. Now, say that $M/x_i,y_i$ has an $N$-minor. By Lemma \ref{w37}, $\si(M/x_i,y_i)\cong\si(M/y_i,y_j)=\si(M/L)$. So, $M/L$ has an $N$-minor and the lemma holds.
\end{proof}

The next lemma has an elementary proof, which is left to the reader.

\begin{lemma}\label{r3}
If $H$ is a rank-$3$ simple matroid with $|E(H)|\ge 4$, then one of the following alternatives holds:
\begin{enumerate}
 \item [(a)] $H$ is the direct sum of a nontrivial line and a coloop.
 \item [(b)] $H$ is connected and has a $4$-circuit.
\end{enumerate}
\end{lemma}

Motivated by Lemma \ref{r3}, for an $(M,N)$-vertbarrier $(C\s,p)$, if $H:=M|(C\s \u p)$ is disconnected, we define the coloop and the line of $H$ respectively as the \defin{coloop} and \defin{line} of $(C\s,p)$ and we say that $(C\s,p)$ is \defin{disconnected}. Otherwise $(C\s,p)$ is said to be \defin{connected}. Combining Lemma \ref{rank3} and Corollary \ref{4-circuit}, we have:

\begin{corollary}\label{connected paradise}
Suppose that $(C\s,p)$ is a connected $(M,N)$-vertbarrier. Then
\begin{enumerate}
\item [(a)] each element of $C\s$ is vertically $N$-contractible in $M$ or
\item [(b)] $M|(C\s\u p)$ is the parallel connection of two nontrivial lines with base point, namely, $b$ and each element of $C\s-b$ is vertically $N$-contractible in $M$.
\end{enumerate}\end{corollary}

\begin{lemma}\label{conf impl carambole}
If $r(M)\ge 4$ and $(C\s,p)$ is a disconnected $(M,N)$-vertbarrier with line $L$, then $L$ is an $N$-filament of $M$ or $L$ contains a vertically $N$-contractible element of $M$.
\end{lemma}%proved

\begin{proof}
Suppose that $L$ contains no vertically $N$-contractible elements of $M$. Let $C\s_1:=C\s$, $y_1:=p$, $L:=\{y_1,\dots, y_n\}$ and $x_1\in C\s_1-L$. 

Let $i\in\{2,\dots,n\}$ with $i\neq j$. By Lemma \ref{w37}, $M/L=\si(M/y_i,y_1)$ is $3$-connected with an $N$-minor. But $\si(M/y_i)$ is not vertically $3$-connected, and, by Lemma \ref{w36}, there is an $(M,N)$-vertbarrier $(C\s_i,y_i)$, containing $y_1$. By orthogonality $L-y_i\cont C\s_i$.

If $(C\s_i,y_i)$ is connected, then, by Corollary \ref{connected paradise} on $C\s_i$, $L$ contains a vertically $N$-contractible element of $M$, a contradiction. Therefore, each $(C\s_i,y_i)$ is a disconnected $(M,N)$-vertbarrier with coloop, namely, $x_i$.

If $L$ intersects a triangle $T\ncont L$, hence, $T\cont \cl_M(C\s_i)$ by orthogonality and, by Lemma \ref{w38}, the elements of $L-T$ are vertically $N$-contractible in $M$. A contradiction to our assumptions. Hence, there is no such triangle. Thus, the line of each $(C\s_i,y_i)$ must be $L$. Now, the result follows from Lemma \ref{carambola potencial}.
\end{proof}

We say that a biweb $x_1,x_2,y_1,y_2,y_3$ is \defin{strict} if there is no element $x_3$ of $M$ such that $x_1,x_2,x_3,y_1,y_2,y_3$ is a carambole of $M$.

\begin{corollary}\label{contractible-b3}
If $r(M)\ge 4$ and $x_1,x_2,y_1,y_2,y_3$ is a strict $N$-biweb of $M$, then $y_3$ is vertically $N$-contractible in $M$.
\end{corollary}
\begin{proof}
Since the biweb is strict, then $\{x_2,y_3\}$ is the unique serial pair of $M\del y_1$, thus, by Corollary \ref{w38-cor}, $\co(M\del y_1)\cong M/y_3\del y_1$ is $3$-connected. But $y_1$ is in parallel in $M/y_3$. So, $\si(M/y_3)\cong\si(M/y_3\del y_1)$ and the result holds.
\end{proof}

In \cite{Ando} is proved that in a $3$-connected graph $G$, each pair of non vertically contractible elements of $M(G)$ incident to a same degree-$3$ vertex is in a triangle and that each degree-$3$ vertex is incident to a vertically contractible element of $M(G)$. Next we generalize this result. This may also be seem as a variation of Tutte's Triangle Lemma.

\begin{proposition}\label{ando}
Suppose that $r(M)\ge 3$ and $C\s$ is a rank-$3$ cocircuit of $M$. If $C\s$ contains two non-vertically contractible elements of $M$, then all non-vertically contractible elements of $M$ in $C\s$ are in a non-trivial line of $M$. Moreover, $C\s$ contains a vertically contractible element of $M$.
\end{proposition}
\begin{proof}\stepcounter{rotcount}
Suppose the contrary. First we prove:
\begin{rot}\label{ando-1}
$C\s$ meets no nontrivial line of $M$.
\end{rot}
\begin{rotproof}
Suppose that $L$ is a non-trivial line of $M$ meeting $C\s$. By Lemma \ref{w38}, all elements of $C\s-L$ are vertically contractible in $M$. But this implies the proposition. So, \ref{ando-1} holds.
\end{rotproof}

If $C\s$ is a triad, then we may assume that $C\s$ contains a pair of non vertically contractible elements of $M$. But, in this case, the dual of Tutte's triangle Lemma contradicts \ref{ando-1}. 

Thus, we may assume that $|C\s|\ge 4$ and, therefore, $C\s$ contains a circuit of $M$. By \ref{ando-1}, $C\s$ contains a $4$-circuit of $M$, and, by Lemma \ref{r3}, $M|C\s$ is a simple rank-$3$ connected matroid. By \ref{ando-1} each element of $C\s$ is in a $4$-circuit of $M|C\s$ and, by Corollary \ref{4-circuit}, all elements of $C\s$ are vertically contractible, a contradiction. So, the proposition holds.
\end{proof}

\section{Lemmas for the proofs}\label{section-lemmas}
In this section we establish Lemmas towards the proof of the main results. We will keep the notations set in the beginning of Section \ref{fancy}.

\begin{lemma}\label{cunningham}(Cunningham~\cite[Proposition 3.2]{Cunningham})
If $y$ is an element of a matroid $H$ other than a coloop and $H\del y$ is vertically $3$-connected, then so is $H$.
\end{lemma}

\begin{corollary}\label{cunningham-cor}
If $M\del x$ is $3$-connected, then each vertically $N$-contractible element of $M\del x$ is vertically $N$-contractible in $M$.
\end{corollary}

\begin{lemma}\label{deletion-carambole}
Suppose that $r(M)\ge 4$, $M\del x$ is $3$-connected and $L$ is an $N$-filament of $M\del x$. Then $\cl_M(L)$ is an $N$-filament of $M$ or $\cl_M(L)$ contains a vertically $N$-contractible element of $M$.
\end{lemma}%revised
\begin{proof}
Since $\si(M\del x/L)\cong \si(M\del x/y_1,y_2)$, then $M\del x/y_1,y_2$ is vertically $3$-connected. Since $x$ is not a coloop of $M/y_1,y_2$. Hence, by Lemma \ref{cunningham}, $M/y_1,y_2$ is a vertically $3$-connected matroid. If $\si(M/y_1)$ is $3$-connected, then the result holds. Assume the contrary. By Lemma \ref{w36}, there is an $(M,N)$-vertbarrier $(D\s,y_1)$ containing $y_2$.

Since $y_2\in D\s$, by orthogonality, $\cl_M(L)-y_1\cont D\s$. If $(D\s,y_1)$ is connected, then Corollary \ref{connected paradise} implies that there is a vertically $N$-contractible element of $M$ in $\cl_M(L)$ and the lemma holds. If $(D\s,y_1)$ is disconnected, then $\cl_M(L)$ is the line of $(D\s,y_1)$ because $\cl_M(L)-y_1\cont D\s$ by orthogonality. The result follows from Lemma \ref{conf impl carambole} in this case.
\end{proof}

\begin{lemma}\label{line}
Suppose that $M/x\ncong U_{2,4}$, $M/x$ is $3$-connected and $L$ is a rank-$2$ set in $M/x$ with at least $4$ elements. Then each $(|L|-1)-$ subset of $L$ contains a deletable or a vertically contractible element of $M$. In particular, $L$ contains a deletable element of $M$ or a pair of vertically contractible elements of $M$. 
\end{lemma}
\begin{proof}
Suppose the contrary. Let $L=\{y_1,\dots,y_n\}$. We may assume that $y_1,\dots,y_{n-1}$ are not deletable nor vertically contractible. For $i\in[n-1]$, $\co(M\del y_i)$ is $3$-connected since $\si(M/y_i)$ is not. But $M\del y_i$ is not $3$-connected, thus, there is a triad $T\s_i$ of $M$ containing $y_i$. Since $T\s_i$ meets $L$, $|L|\ge 4$ and $M/x$ is $3$-connected and not isomorphic to $U_{2,4}$, hence $T\s_i$ is not a triad of $M/x$. So, $x\in T\s_i$. 

Suppose first that for some $1\le i<j\le n-1$, $T\s_i\i T\s_j\neq \{x\}$, then $2\le |(T\s_i\u T\s_j)-x|\le 3$. If $|(T\s_i\u T\s_j)-x|=3$ , then $(T\s_i\u T\s_j)-x$ is a triad of $M/x$ meeting $L$, a contradiction. If $|(T\s_i\u T\s_j)-x|=2$, then $T_i\s=T_j\s=\{x,y_i,y_j\}$. If for some $k\in [n]-\{i,j\}$, $C:=\{x,y_i,y_j,y_k\}$ is a circuit of $M$, then by Lemma \ref{4-circuit} on $T\s_i$ and $C$, $y_i$ and $y_j$ are vertically contractible and the Lemma holds. So, $\{y_i,y_j,y_k\}$ is a triangle for each $k\in [n]$ and $L$ is a line of $M$, which implies that the elements of $L$ are deletable and also implies the Lemma.

So, we may assume that $T\s_i\i T\s_j=\{x\}$ if $1\le i<j\le n-1$. This implies that $T\s_i\Delta T\s_j$ is a cocircuit of $M$ and, therefore, of $M/x$. By orthogonality, $|L-(T\s_i\Delta T\s_j)|\le 1$. Since, $y_k\notin T\s_i\Delta T\s_j$ for $k\in [n-1]-\{i,j\}$, then $n=4$ and $y_4\in T\s_i\Delta T\s_j$ for each $1\le i<j<n-1=3$. So, $y_4$ is in two of $T\s_1$, $T\s_2$ and $T\s_3$. But, $T\s_i\i T\s_j=\{x\}$ if $1\le i<j\le n-1$. A contradiction.
\end{proof}

\begin{lemma}\label{upcont}
Suppose that $r(M)\ge4$, $M/x$ is $3$-connected and $M$ has no $N$-deletable elements. If $L$ is an $N$-filament of $M/x$, then $L$ is an $N$-filament of $M$ or $L$ contains a pair of vertically $N$-contractible elements of $M$.
\end{lemma}%???? is this rank optimum??

\begin{proof}\stepcounter{rotcount}
Suppose that the result does not hold. By Lemma \ref{line}, $|L|=3$. So, $L$ a filament of $M/x$. We may assume that $y_1$ and $y_2$ are not vertically contractible in $M$. By Proposition \ref{ando} on $C\s_3$, $y_1$ and $y_2$ are in a triangle $T$ of $M$. As $M/x$ is $3$-connected, hence $x\notin T$ and $T$ is a triangle of $M/x$. By Corollary \ref{triangle-carambole}, $T=L$. The result follows from Lemma \ref{carambola potencial}.
\end{proof}

In some cases, it is easier to prove that an element $p$ is spanned by vertically $N$-contractible elements instead of proving that $p$ is vertically $N$-contractible itself. In this case, some exchanges to get a desired vertically $N$-contractible element may be applied (we will do it further, in Lemma \ref{lifter}). We say that an element $p\in E(M)$ is \defin{$N$-replaceable} in $M$ if $p$ is spanned by a set of vertically $N$-contractible elements of $M$. For $S\cont E(M)$, we say that $p$ is \defin{$(S,N)$-replaceable} if there is a set $I$ of vertically $N$-contractible elements of $M$ such that $p\in \cl_M(S\u I)-\cl_M(S)$. 

%\begin{corollary}\label{4-circuit-replaceable}If $(C\s,p)$ is a connected $(M,N)$-vertbarrier, then each element of $C\s\u p$ is $N$-replaceable in $M$.\end{corollary}

\begin{lemma}\label{no deletable}
Suppose that $M\ncong U_{2,4}$ has no $N$-deletable elements. If $\si(M/x,p)$ is $3$-connected with an $N$-minor, then $p$ is $(\{x\},N)$-replaceable in $M$ or $M$ has an $N$-biweb $x,x_2,p,p_2,p_3$.
\end{lemma}

\begin{proof}
Suppose the contrary. So, $p$ is not $N$-replaceable in $M$, because, otherwise, since $p\notin\cl_M(\{x\})$, $p$ would be $(\{x\},N)$-replaceable in $M$. In particular $p$ is not vertically contractible. By Lemma \ref{w36}, $r(M)\ge 4$ and there is an $(M,N)$-vertbarrier $(C\s,p)$ containing $x$. If $(C\s,p)$ is connected, then by Lemma \ref{connected paradise} $p$ is spanned by vertically $N$-contractible elements of $M$ and $p$ is $N$-replaceable, a contradiction. So, $(C\s,p)$ is disconnected, with line, say, $L$. By Lemma \ref{w37}, $\si(M/x,p)\cong \si(M/L)$ is $3$-connected with an $N$-minor. Then $M\del p$ has an $N$-minor. As, $\si(M/p)$ is not $3$-connected, then $\co(M\del p)$ is $3$-connected. But $p$ is not $N$-deletable, and, therefore, $p$ is in a triad $T\s$ of $M$. So, $L$ is a triangle because $L$ meets $T\s$. This also implies that $C\s$ is a triad. Since $C\s$ and $T\s$ are distinct triads meeting $L$, those triads are not in a same coline of $M$. By orthogonality, $|C\s\i L|=|T\s\i L|=2$. So, we may write $L:=\{p,p_2,p_3\}$, $C\s:=\{x,p_2,p_3\}$ and $T\s:=\{x_2,p,p_3\}$. As $\si(M/x,p)\cong \si(M/L)$, this proves the lemma.
\end{proof}

Lemma \ref{cocircuit-union} yields:

\begin{corollary}\label{biweb-rank-3}
If $r(M)\ge4$ and $W$ is a biweb of $M$, then $r_M(W)\ge 4$.
\end{corollary}

\begin{lemma}\label{biweb-notriangle}
If $r(M)\ge 4$, $M$ has no $N$-deletable elements and $W=x_1,x_2,y_1,y_2,y_3$ is an $N$-biweb of $M$, then $\{y_1,y_2,y_3\}$ is the unique triangle of $M$ intersecting $W$.
\end{lemma}
\begin{proof}
Suppose for a contradiction that $S$ is a triangle of $M$ intersecting $W$ other than $T:=\{y_1,y_2,y_3\}$ Then, by orthogonality with $T\s_i:=\{x_i,y_{3-i},y_3\}$ for $i=1,2$, we may assume that $x_1\in S$. If $y_3\in S$, then $S\cont W$ by orthogonality with $T\s_2$. But this implies that $r_M(W)\le 3$. A contradiction to Lemma \ref{biweb-rank-3}. So, $y_2\in S$ by orthogonality with $T\s_1$. Hence, $W$ is part of a maximal fan $F$ of $M$ containing $S$. As $M\del y_1$ has an $N$-minor and $y_1$ is an inner spoke of $F$, thus each deletion of a spoke of $F$ in $M$ has an $N$-minor. But $M$ has no $N$-deletable elements, so the extremes of $F$ are triads. Moreover, if $T\s$ is a triad in $F$ intersecting $x_i$ but different from $T\s_i$, then $T\s$ is a rank-$2$ cocircuit of $M/T$, which is vertically $3$-connected with rank at least $3$, a contradiction to the vertical $3$-connectivity of $M/T$. So, $F=W$ and the lemma holds.
\end{proof}

\begin{lemma}\label{contractible-triangle}
Suppose that $r(M)\ge 4$, $W=x_1,x_2,y_1,y_2,y_3$ is an $N$-biweb of $M$ and $M$ has no $N$-deletable elements. Then $M/\{y_1,y_2,y_3\}$ is $3$-connected.
\end{lemma}
\begin{proof}Write $T:=\{y_1,y_2,y_3\}$. By the definition of biweb, we just have to prove that $M/T$ is simple. Suppose the contrary and let $C$ be a circuit of $M$ such that $1\le C-T\le 2$. By Lemma \ref{biweb-notriangle}, $|C|=4$. Therefore, $|C\i T|=|C-T|=2$. By Corollary \ref{biweb-rank-3}, $C\ncont W$. So, we may assume that $x_2\notin C$. But $C$ meets $\{y_1,y_3\}$ because $|C\i T|=2$. By orthogonality with $\{x_2,y_1,y_3\}$, $C\i T=\{y_1,y_3\}$. So, $y_2\notin C$. By orthogonality with $\{x_1,y_2,y_3\}$, $x_1\in C$.
Let $D$ be a circuit of $M$ contained in $(C\u T)-y_3$. By orthogonality with $\{x_2,y_1,y_3\}$, $y_1\notin D$. As $|D|\ge 3$, hence $D=(T\u C)-\{y_1,y_3\}$ and $|D|=3$. A contradiction to Lemma \ref{biweb-notriangle}.\end{proof}

\begin{lemma}\label{up-fil-biweb}
Suppose that $r(M)\ge 4$, $T$ is the triangle of an $N$-biweb of $M$ and $M$ has no $N$-deletable elements. If $L$ is an $N$-filament of $M/T$, then $L$ is an $N$-filament of $M$ or $L$ contains an $N$-contractible element of $M$.
\end{lemma}

\begin{proof} Let $W=x_1,x_2,y_1,y_2,y_3$ be an $N$-biweb of $M$ with $T:=\{y_1,y_2,y_3\}$. If $L$ is not a line of $M$, then $\{L,T\}$ is not free in $M$. So, there is $C\in\C(M|L\u T)$ meeting both $L$ and $T$. By orthogonality with $\{x_1,y_2,y_3\}$ and $\{x_2,y_1,y_3\}$, $\{x_1,x_2\}$ meets $C$ and, therefore, $L$. By Corollary \ref{w38-cor} and Lemma \ref{biweb-notriangle}, $x_1$ and $x_2$ are $N$-contractible in $M$. This implies the Lemma in this case.

Now, assume that $L$ is a line of $M$. Say that $L$ is part of a carambole $K$ of $M/T$. The cocircuits of $K$ in $M/T$ are also cocircuits of $M$. The result follows from Lemma \ref{carambola potencial}.
\end{proof}

\begin{lemma}\label{remendo}If $r(M)\neq 5$ and $T$ is a triangle of an $N$-biweb of $M$, then each vertically $N$-contractible element of $si(M/T)$ is $(T,N)$-replaceable in $M$.
\end{lemma}

\begin{proof}
Let $W=x_1,x_2,y_1,y_2,y_3$ be a biweb of $M$ having $T$ as triangle. Suppose for a contradiction that $p$ is a vertically $N$-contractible element of $\si(M/T)$ which is not $N$-replaceable in $N$. Note that that $p\notin \cl_M(T)$. As $x_1$ and $x_2$ are vertically $N$-contractible by Corollary \ref{w38-cor}, it follows that $p\notin \cl_M(W)$. Now,  $\si(M/T\cup p)$ is $3$-connected but $\si(M/p)$ is not. If $r(M)\le 3$, then $M/p$ is trivially vertically $3$-connected. So, we may assume that $r(M)\ge 4$. By Corollary \ref{biweb-rank-3} and since $p\notin cl_M(W)$, it follows that $r_M(W\u p)=5$ and, therefore, $r(M)\ge 6$, as $r(M)\neq 5$ by assumption. Let $\{A,B\}$ be a vertical $2$-separation of $M/p$ such that $|B\i T|\ge 2$ with $B$ maximal. Note that $A$ is a minimal vertical $2$-separating set of $M$. By Lemma \ref{vertsep} for $Y=\emptyset$, $B$ is closed, and, therefore, $T\cont B$.

Let us check that $r_{M/T\cup p}(B-T)\ge 2$. First suppose that $x_1,x_2\in B$, then $W\cup p\cont B\cup p$ and, therefore $r_{M/p}(B)\ge 4$ and this implies that $r_{M/T\cup p}(B-T)\ge 2$. Now, consider the case that $x_i\in A$ for some $i\in \{1,2\}$. By Lemma \ref{2-cocirc}, for $z=x_i$, it follows that $A$ is a rank-$2$ cocircuit and $B$ is a hyperplane of $M/p$. Thus $r_{M/p}(B)=r(M/p)-1\ge 4$ and, as a consequence, $r_{(M/T\u p)}(B-T)\ge 2$ in all cases.

Now, note that:
\begin{equation}\label{pre rank 5 - eq1}
\begin{array}{rcl}
r_{(M/T\u p)}(B-T)+r_{(M/T\u p)}(A)&=&r_{M/p}(B)-2+r_{M/p}(A)+(r_{(M/T\u p)}(A) - r_{M/p}(A))\\
& \le & r(M/p)-1+(r_{(M/T\u p)}(A) - r_{M/p}(A))\\
& = & r(M/T\u p)+1+(r_{(M/T\u p)}(A) - r_{M/p}(A)).
\end{array}
\end{equation}
As $T$ meets two triads, it follows that $T$ is a flat of $M/p$ disjoint from $A$. Hence, $r_{M/p}(T\u A)\ge 3$ and $r_{(M/T\u p)}(A)\ge 1$. Since $M/(T\u p)$ is vertically connected, the right side of \eqref{pre rank 5 - eq1} is at least $r(M/T\u p)+1$, so $r_{(M/T\u p)}(A)=r_{M/p}(A)\ge 2$. But $r_{(M/T\u p)}(B-T)\ge 2$ and, therefore, \eqref{pre rank 5 - eq1} contradicts the vertical $3$-connectivity of $M/(T\u p)$.
\end{proof}

\begin{lemma}\label{rank5}
If $M$ has no deletable elements, $r(M)=5$ and $M$ has a biweb, then $M$ has a filment with $3$ elements or a $4$-independent set of vertically contractible elements.
\end{lemma}

\begin{proof}
Let $W:=x_1,x_2,y_1,y_2,y_3$ be a biweb of $M$. We may assume that $W$ is strict. By Lemma \ref{contractible-b3} and Corollary \ref{w38}, the elements of $I:=\{x_1,x_2,y_3\}$ are vertically contractible in $M$. By Corollary \ref{biweb-rank-3}, $I$ is independent and $r_M(W)=4=r(M)-1$. Hence, $E(M)-W$ contains a cocircuit $C\s$. As $W$ is an union of two triads, $r_M(C\s)\le r_M(E(M)-W)\le r(M)-2=3$. So $C\s$ is a rank-$3$ cocircuit of $M$. By Proposition \ref{ando}, $C\s$ contains a vertically contractible element $z$ of $M$. Now, $I\u\{z\}$ is a $4$-independent set of vertically contractible elements of $M$.
\end{proof}

The next Lemma has an elementary proof, which will be omitted.

\begin{lemma}\label{free-closure}
If $\{A_1,\dots,A_n\}$ is a free family of $M$ and, for each $i=1, \dots, n$, $B_i\cont \cl_M(A_i)$, then $\{B_1,\dots,B_n\}$ is a free family of $M$.
\end{lemma}

\begin{lemma}\label{free-contraction}
If $\{A_1,\dots,A_n\}$ is a free family of $M/X$ and $r_{M/X}(A_i)=r_M(A_i)$ for each $i\in[n]$, then $\{X,A_1,\dots,A_n\}$ is a free family of $M$.
\end{lemma}
\begin{proof}
Note that:
\begin{eqnarray*}
r_M(X\u A_1\u\cdots\u A_n) &\le& r_M(X)+r_M(A_1)+\cdots+ r_M(A_n)\\
&=&r_M(X)+r_{M/X}(A_1)+\cdots+r_{M/X}(A_n)\\
&=&r_M(X)+r_{M/X}(A_1\u\cdots\u A_n)\\
&=&r_M(X\u A_1\u\cdots\u A_n)
\end{eqnarray*}
Thus equality holds above. This  implies the lemma.
\end{proof}

We say that a singleton subset of $E(M)$ is {\bf $(X,N)$-replaceable} in $M$ if its element is $(X,N)$-replaceable in $M$.

\begin{lemma}\label{lifter}
Suppose that $M/X\del Z$ is simple with an $N$-minor and $\{X_1,\dots,X_ n\}$ is a free family of $M/X\del Z$. Suppose also that, for each $X_i$, one of the following alternatives holds:
\begin{enumerate}
\item[(a)] For some $2$-subset $Y_i\cont X_i$, $\cl_M(Y_i)$ contains an $(X,N)$-replaceable element or an $N$-filament of $M$.
\item[(b)] $X_i$ is an $(X,N)$-replaceable singleton set of $M$.
\end{enumerate}
Then, $M$ has a free family $\{X,Z_1,\dots,Z_n\}$ such that, for $k\in[n]$, $Z_k$ is an $N$-filament or a vertically $N$-contractible singleton set of $M$.
\end{lemma}
\begin{proof}By Lemma \ref{free-closure}, we may define a free family $\{X,Y_1,\dots,Y_n\}$ of $M/X\del Z$, choosing $Y_i$ as a $2$-subset of $X_i$ according to (a), provided (a) holds, and choosing $Y_i:=X_i$ otherwise. As $M$ and $M/X\del Z$ are simple, hence $1\le r_{M\del Z}(Y_i)=|Y_i|=r_{M/X\del Z}(Y_i)\le 2$ for each $i$. By Lemma \ref{free-contraction}, $\{X,Y_1,\dots,Y_n\}$ is a free family of $M\del Z$, and, therefore, of $M$. Next, for each $i$, define $W_i$ as a subset of $\cl_M(Y_i)$ that is an $N$-filament or an $(X,N)$-replaceable singleton set of $M$. By Lemma \ref{free-closure}, $\F_0:=\{X,W_1,\dots,W_n\}$ is free. Now, for $k\in[n]$, let us define $Z_k$ inductively in such a way that each $\F_k:=\{X,Z_1,\dots,Z_k,W_{k+1},\dots,W_n\}$ is free and $\F_n$ satisfies the Lemma.

If $W_k$ is an $N$-filament of $M$ or a vertically $N$-contractible singleton set of $M$, simply define $Z_k:=W_k$. Otherwise, $W_k$ is an $(X,N)$-replaceable but not vertically $N$-contractible singleton set $\{w\}$ of $M$. Let $I$ be a set of vertically $N$-contractible elements and $C$ a circuit of $M$ such that $w\in C\cont (X\u I)\u w$. As $\F_{k-1}$ is free and $w\in W_k$, hence $w\notin F:=\cl_M(X\u Z_1\u\dots\u Z_{k-1}\u W_{k+1}\u\dots\u W_n)$ and, therefore, there is an element $z\in C-(F\u w)\cont I$. Since $z\in I$, then $z$ is vertically $N$-contractible. Define $Z_k=\{z\}$. Since $\F_{k-1}-\{W_k\}$ is free and $z\notin F$, then $\F_k=(\F_{k-1}-\{W_k\})\u\{Z_k\}$ is free. The family $\F_n$ satisfies the Lemma.\end{proof}

\section{Proofs for the main results}
\label{sec-proofs}
\begin{proofof}\emph{Proof of Theorem \ref{main-theorem}: }Suppose that $(M,N)$ is a counter-example to the theorem minimizing $|E(M)|$. By Theorem \ref{whittle-teo}, the theorem holds for $k\le 4$. So, $k\ge 5$. It is straightforward to verify that $M$ is not a wheel or whirl: in such case $N$ would be also a wheel or whirl respectively or $N\le U_{2,4}$, implying that the elements in the rim of $M$ are vertically $N$-contractible. By Seymour's Splitter Theorem, exactly one of the following cases occur:
\begin{enumerate}
\item [(i)] $M$ has an $N$-deletable element.
\item [(ii)] $M$ has an $N$-biweb and no $N$-deletable elements.
\item [(iii)] $M$ has an $N$-contractible element $x$ and (i) and (ii) do not occur.
\end{enumerate}

Next, we complete the proof for each case.

{\it Case (i):} Define $X:=\emp$ and $Z$ as an $N$-deletable singleton set of $M$. Let $\F:=\{X_1,\dots,X_n\}$ be a family satisfying Theorem \ref{main-theorem} for $M\del Z=M\del Z/X$. By Corollary \ref{cunningham-cor} and Lemma \ref{deletion-carambole}, $\F$ satisfies the hypothesis of Lemma \ref{lifter}. The family $\{Z_1,\dots,Z_n\}$ obtained from Lemma \ref{lifter} satisfies the theorem since $r(M)=r(M\del Z)$ in this case.

{\it Case (ii):} If $r(M)=5$, then $r(N)=0$ and the theorem follows from Lemma \ref{rank5}. Assume that $r(M)\ge 6$ for this case. Let $W$ be an $N$-biweb of $M$ with triangle, namely, $T$. By Lemma \ref{contractible-triangle}, $M/T$ is $3$-connected with an $N$-minor. Let $\F:=\{X_1,\dots,X_n\}$ be a family satisfying Theorem \ref{main-theorem} for $M/T$. Define $X:=T$ and $Z:=\emp$. The hypothesis of Lemma \ref{lifter} are satisfied because of Lemma \ref{remendo} and Lemma \ref{up-fil-biweb}. Let $\{X,Z_1,\dots,Z_n\}$ be a free family of $M$ as in Lemma \ref{lifter}. If $W$ is part of a carambole, then $X$ is an $N$-filament and the theorem holds. Otherwise, if $W$ is a strict biweb, then, by Corollary \ref{contractible-b3}, $X$ contains a vertically $N$-contractible element $x$ of $M$ and $\{\{x\},Z_1,\dots,Z_n\}$ satisfies the theorem.

{\it Case (iii):} Define $Z:=\emp$ and $X:=\{x\}$. Let $\F:=\{X_1,\dots,X_n\}$ be a family satisfying Theorem \ref{main-theorem} for $M/x=M/X\del Z$. By Lemma \ref{no deletable}, all vertically $N$-contractible elements of $M/x$ are $(X,N)$-replaceable in $M$. So, by Lemma \ref{upcont}, the hypothesis of Lemma \ref{lifter} are once more satisfied. The family $\{X,Z_1,\dots,Z_n\}$, obtained from Lemma \ref{lifter}, satisfies the theorem as before. This completes the proof.
\end{proofof}

\begin{proofof}\emph{Proof of Proposition \ref{independent-seeds}: }
By Proposition \ref{1st-prop}, $|X|=n$, where $n$ is the size of the filament. Moreover,  by orthogonality with the cocircuits of the carambole, $X$ may not contains circuits. If some triangle meets $X$, then by orthogonality, it must meet $L$, a contradiction to Corollary \ref{triangle-carambole}. Now, the elements of $X$ are $N$-contractible because of Lemma \ref{w38}. So, the proposition holds.
\end{proofof}

\begin{proofof}\emph{Proof of Corollary \ref{k4}: }
The proof is analogous to the proof of Theorem \ref{main-theorem}. Consider a counter-example minimizing $|E(M)|$. The proof in case (i) is the same. In case (ii), by Corollaries \ref{w38-cor} and \ref{contractible-b3}, we have the result. In case (iii), instead of the minimality of $M$, we use Theorem \ref{whittle-teo} for $k=3$ to obtain $\F$ and, in the same way, finish the proof.\end{proofof}

\begin{proofof}\emph{Proof of Corollary \ref{k5}: }
Consider the family $\{X_1,\dots,X_4\}$ given by Theorem \ref{main-theorem}. We may assume that $X_1$ is an $N$-filament with more than $3$ elements. By Proposition \ref{independent-seeds}, $M$ has a $4$-independent set of vertically $N$-contractible elements.
\end{proofof}

With the lemmas we established here, it is possible to give an alternative proof for Theorem \ref{whittle-teo}. This is interesting for the self-sufficiency of this work.\medskip

\begin{proofof}{\emph{Proof of Theorem \ref{whittle-teo}: } }
Suppose that $M$ and $N$ contradict the theorem, minimizing $(k, |E(M)|)$ lexicographically. If $k=1$, the result follows directly form Seymour's Splitter Theorem and Bixby's Lemma (this is made with details in \cite[Lemma 12.3.11]{Oxley}). So, $k\ge 2$. If $r(M)\le 3$, then $r(N)\le 1$ and the result is trivial. Suppose that $r(M)\ge 4$. By Proposition \ref{independent-seeds}, $M$ has no $N$-caramboles. By Corollaries \ref{w38-cor} and \ref{contractible-b3}, $M$ has no $N$-biwebs. By Corollary \ref{cunningham-cor} and the minimality of $|E(M)|$, $M$ has no $N$-deletable elements. By Seymour's Splitter Theorem $M$ has an $N$-contractible element $x$. By the minimality of $k$, there is a free family $\F$ with $k-1$ vertically $N$-contractible singleton sets of $M/x$. Now, the proof is finished as in case (iii) of the proof of Theorem \ref{main-theorem}.
\end{proofof}

\section*{Acknowledgments}

The Author thanks to CNPq for the financial support, to Manoel Lemos for the help with raising funds with CNPq and to James Oxley for noticing the relations between carambole and generalized delta-wye exchanges.

\providecommand{\bysame}{\leavevmode\hbox to3em{\hrulefill}\thinspace}
\providecommand{\MR}{\relax\ifhmode\unskip\space\fi MR }
\providecommand{\MRhref}[2]{%
  \href{http://www.ams.org/mathscinet-getitem?mr=#1}{#2}
}
\providecommand{\href}[2]{#2}

\end{document}